\documentclass[11pt]{article}
\usepackage{mypackages}
\usepackage{makeidx}

\usepackage[T1]{fontenc}
\usepackage{titlesec}
\usepackage{url}
\newcommand{\kap}{\kappa}

\title{\textbf{Scaling laws and exact results in deterministic flows}}

\author{{\textsc{matthew novack}}}

\date{}

\begin{document}

\maketitle

\begin{abstract}
In this note, we address the validity of certain exact results from turbulence theory in the deterministic setting.  The main tools, inspired by the work of Duchon-Robert~\cite{DuchonRobert} and Eyink~\cite{Eyink}, are a number of energy balance identities for weak solutions of the incompressible Euler and Navier-Stokes equations. As a consequence, we show that certain weak solutions of the Euler and Navier-Stokes equations satisfy deterministic versions of Kolmogorov's $\sfrac 45$, $\sfrac 43$, $\sfrac{4}{15}$ laws.  We apply these computations to improve a recent result of Hofmanova et al.~\cite{HPZZ}, which shows that a construction of solutions of forced Navier-Stokes due to Bru\`e et al.~\cite{BCCDLS} and exhibiting a form of anomalous dissipation satisfies asymptotic versions of Kolmogorov's laws.  In addition, we show that the globally dissipative 3D Euler flows recently constructed by Giri, Kwon, and the author~\cite{GKN23} satisfy the local versions of Kolmogorov's laws.
\end{abstract}

\newcommand{\AAA}{\mathcal{A}}
\newcommand{\dee}{\mathrm{d}}
\newcommand{\MMM}{\mathcal{M}}

\section{Introduction}
The purpose of this note is to determine whether some recent deterministic constructions of weak solutions to the Euler and Navier-Stokes equations satisfy appropriate versions of Kolmogorov's famous $\sfrac 45$, $\sfrac 43$, and $\sfrac{4}{15}$ laws.  There have been a number of results over the past few decades addressing such questions, which we review below.  However, we were unable to find precise statements in the literature which provided satisfactory answers to our motivating questions, the first of which is the following: when do weak solutions of the 3D incompressible Euler equations satisfy the distributional equality
\begin{align}
    \pa_t &\left( \frac{|u|^2(t,x)}{2} \right) + \div \left( u(t,x) \left( \frac{|u|^2(t,x)}{2} + p(t,x) \right)  \right) \notag\\
    &\qquad = \lim_{\ell\rightarrow 0} \frac{3}{4\ell} \dashint_{\S^2} y_j \left( u^j(t,x+y) - u^j(t,x) \right) |u(t,x+y) - u(t,x)|^2 \, \dee y \, ?  \label{local43}
\end{align}
Following Duchon and Robert~\cite{DuchonRobert} and Eyink~\cite{Eyink}, we refer to~\eqref{local43} as a ``deterministic, local $\sfrac 43$ law.''  It is well-known that the left-hand side is equal to the Duchon-Robert distribution~\cite{DuchonRobert}
$$ D[u](t,x) = \lim_{\ell\rightarrow 0} D_\ell[u](t,x) = \lim_{\ell\rightarrow 0} \frac 14 \int_{\T^3} \pa_j \varphi_\ell(y) \left( u^j(t,x+y) - u^j(t,x) \right) |u(t,x+y) - u(t,x)|^2 \, \dee y \, , $$
where $\varphi_\ell$ is a mollification kernel at scale $\ell$. Concerning~\eqref{local43}, however, both~\cite{DuchonRobert, Eyink} assert that it holds \emph{assuming} the distributional limit of the right-hand side exists as $\ell\rightarrow 0$. As it turns out, only the mild assumption $u\in L^3_{t,x}$ is needed in order to define the Duchon-Robert distribution for Euler weak solutions, and we shall show in Theorem~\ref{thm:main} that the same assumption suffices to ensure that~\eqref{local43} holds, with similar conclusions in the cases of the $\sfrac 45$ and $\sfrac{4}{15}$ laws.  The proof of this fact follows the general strategy of Duchon-Robert/Eyink; our contribution is to remove the conditionality on the existence of the limit. We do so by adding an intermediate step in the proof, in which we choose a sequence of mollification kernels $\{\varphi_{\ell,\gamma}\}_{\gamma > 0}$, with gradients supported in a neighborhood of size $\ell\gamma$ around the sphere $\S^{d-1}_\ell$ of radius $\ell$, and pass to the limit $\gamma\rightarrow 0$.  Upon doing so, we obtain an energy balance for $\ell>0$ which allows us to pass to the limit $\ell\rightarrow 0$ and obtain precisely~\eqref{local43}. Our interest in this question stems from recent joint work with Giri and Kwon~\cite{GKN23b, GKN23}, in which we construct energy-dissipating weak solutions to the 3D Euler equations with $D[u] \geq 0$ and $u\in C^0_t B^{\sfrac 13-}_{3,\infty}(\T^3)$.  Conversely, it is easy to check that if an Euler weak solution $u\in L^3_t B_{3,\infty}^+(\T^d)$, then $D[u]=0$, and thus $u$ conserves energy.  We refer also to work of Eyink~\cite{Eyink94}, Constantin, E, and Titi~\cite{CET}, and Cheskidov, Constantin, Friedlander, and Shvydkoy~\cite{CCFS} for different proofs of conservation of energy under various conditions.

The second motivation for this work lies in recent results of Bru\`e, Colombo, Crippa, De Lellis, and Sorella~\cite{BCCDLS} and Hofmanova, Pappalettera, Zhu, and Zhu~\cite{HPZZ}. The former constructs a sequence of Leray weak solutions $\{u_\nu\}$ to the forced\footnote{The smooth forcings $f_\nu$ depend on $\nu$ and lose some regularity in the limit $\nu\rightarrow 0$, although they remain bounded in $L^{1+}_t C^{0+}_x$; this bound rules out the possibility of anomalous dissipation for the heat equation~\cite{BCCDLS}.} 3D Navier-Stokes equations, uniformly bounded in $L^3_t C^{\sfrac 13-}_x$, for which
\begin{equation}\label{ad}
  \varepsilon = \lim_{\nu\rightarrow 0}{\varepsilon_\nu} = \lim_{\nu\rightarrow 0} \nu \int_0^1 \int_{\T^3} \left| \nabla u_\nu \right|^2 > 0 \, . 
\end{equation}
The latter shows that this sequence of solutions satisfies the asymptotic relation
\begin{align}
\lim_{\ell_I \rightarrow 0} \limsup_{\nu\rightarrow 0} \sup_{\ell\in[\ell_D, \ell_I]} \int_0^1 \, \bigg{[} &\int_0^t \frac{1}{\ell} \int_{\T^3} \dashint_{\S^2} y^j \left( u_\nu^j(r,x+\ell y) - u_\nu^j(r,x) \right) \notag\\
&\qquad \qquad |u_\nu(r,x+\ell y) - u_\nu(r,x)|^2 \, \dee y  \, \dee r + \frac{4}{3} \varepsilon_\nu(t) \bigg{]}^p \,   \dee t = 0 \, , \label{43?}
\end{align}
where $p\in[1,\infty)$, $\ell_D = \nu^{\sfrac 12-}$, and $\varepsilon_\nu(t)$ replaces $1$ with $t$ in~\eqref{ad}. The inspiration for these works is Kolmogorov's famous 1941 phenomenological theory of turbulence \cite{K2, K3, K1}, framed in the context of weak solutions of the 3D Euler or Navier-Stokes equations
\begin{equation}\label{eqn:NSE:Euler}
    \begin{cases}
    \pa_t u + (u\cdot \na) u + \na p = \nu \Delta u + f \\
    \div \, u =0 \\
    u|_{t=0} = u_0  \, .
    \end{cases}
\end{equation}
Kolmogorov's theory assumes a non-vanishing energy dissipation rate $\varepsilon$ as in~\eqref{ad}, as well as homogeneity, isotropy, and self-similarity of velocity increments~\cite{Frisch}.  Under these assumptions, Kolmogorov predicts a scaling relation similar to~\eqref{43?}, \emph{which however should hold for  $\ell\in [\nu^{\sfrac 34}, \ell_I]$}. In other words, the dissipative length scale $\ell_D$ used in~\cite{HPZZ} is not the one predicted by Kolmogorov's theory.  Therefore, we revisit this question and prove in Corollary~\ref{local:limit} that in fact~\eqref{43?} holds for $\ell_D=\nu^{\sfrac 34-}$ for the solutions constructed in~\cite{BCCDLS}.  More generally, we show that the choice of dissipative length scale $\ell_D$ can be made using knowledge of the \textit{second order structure function exponent}, or in this setting, the number $\zeta_2$ such that the sequence $u_\nu$ enjoys uniform bounds in $L^2_t B^{\sfrac{\zeta_2}{2}}_{2,\infty,x}$. Our proof is a straightforward application of the energy balance identities we derive in the course of proving Theorem~\ref{thm:main} and is similar in spirit to~\cite{HPZZ} (which itself relies on earlier work of Bedrossian, Coti Zelati, Punshon-Smith, and Weber in the stochastic setting~\cite{BCZ} and is essentially an application of the Kolmogorov-K\'arm\'an-Howarth/K\'arm\'an-Howarth-Monin identity).  We however make proper use of the uniform $L^3_t C^{\sfrac 13-}_x$ regularity of the example from~\cite{BCCDLS} in order to obtain the correct dissipative length scale. We can similarly treat the versions of~\eqref{43?} for the $\sfrac{4}{15}$ and $\sfrac 45$ laws (the former of which is not mentioned in~\cite{BCZ, HPZZ}, although it may be treated using the same ideas).

In order to state our main theorems, we first provide a few definitions.  We define weak solutions to~\eqref{eqn:NSE:Euler} according to the integral equality
\begin{align}
    &\int_{\T^d} \left[ \phi^i(0,x) u_0^i(x) - \phi^i(T,x) u^i(T,x) \right] \, \dee x + \int_{\T^d\times[0,T]} \phi^i(t,x) f^i(t,x) \, \dee t \, \dee x \notag\\
    &= -\int_{\T^d\times[0,T]} \left[ (\pa_t \phi^i u^i)(t,x) - ( \pa_j \phi^i u^i u^j )(t,x) - (\pa_i \phi^i p)(t,x) + \nu (\pa_j \phi^i \pa_j u^i)(t,x) \right] \, \dee t \, \dee x \label{weak:formulation}
\end{align}
for all $\phi^i \in C^\infty\left([0,T]\times \T^d\right)$. If $\nu=0$, we only require $u\in C^0 \left([0,T];L^2(\T^d)\right)$, while if $\nu>0$, we additionally require $u \in L^2\left([0,T]; \dot H^1(\T^d)\right)$.  We define the symmetric tensors
\begin{equation}\label{eq:tensors}
    T^{ij}_I(y) : = \delta^{ij} \, , \qquad T^{ij}_L(y) : = \frac{y^iy^j}{|y|^2} \, , \qquad T^{ij}_T(y) :=  \left( \delta^{ij} - \frac{y^i y^j}{|y|^2} \right) \, . 
\end{equation}
Our first theorem gives three energy balance identities where the dissipation measures are defined using integrals over spheres, making rigorous the claims regarding these measures in~\cite{DuchonRobert, Eyink} for the 3-dimensional case; we refer also to Eyink's notes~\cite{EyinkNotes} for the case of general dimensions $d\geq 2$.

\begin{theorem}[\textbf{Energy balance identities}]\label{thm:main}
    Let $d \geq 2$, and let $u(t,x):[0,T]\times \T^d\rightarrow \R^d$ be a weak solution of the $d$-dimensional Euler or Navier-Stokes equations with $u\in L^3\left([0,T]\times \T^d\right) \cap C^0 \left( [0,T];L^2(\T^d)\right)$; if $\nu \neq 0$ we additionally require $u\in L^2\left([0,T];\dot H^1(\T^d) \right)$.  Assume also that $f \in L^{\sfrac 32}\left([0,T]\times\T^d\right)\cup L^1\left([0,T];L^2(\T^d) \right)$ and $u_0 \in L^2(\T^d)$.  Then the following balance laws hold in the sense of distributions:
    \begin{align}
        \pa_t \left(|u|^2 \right) &+ \partial_j \left( u^j |u|^2 + 2 p u^j \right) - 2 u^i f^i - \nu \left( \Delta(|u|^2) - 2\pa_k u^i \pa_k u^i \right) = -2 D_{\bullet}[u] \, , \label{eq:main:balance:thm}
    \end{align}
    where the distributions $D_{\bullet}[u](t,x)$ are defined for $\bullet=I,L,T$ by the formulas
    \begin{subequations}\label{formulas}
    \begin{align}
        % D_{I,\ell,\gamma}[u](t,x) &= \frac 14 \int_{\T^d} \pa_j \varphi_{\ell,\ga}(y) \left( u^j(t,x+y) - u^j (t,x) \right) \left| u(t,x+y) - u(t,x) \right|^2 \, \dee y \, , \qquad \gamma>0 \\ 
       -\frac 4 d \cdot D_{I}[u] &= \lim_{\ell\rightarrow 0} \frac{1}{\ell} \dashint_{\S^{d-1}}  y^j \left( u^j(t,x+\ell y) - u^j(t,x) \right) \left| T_I(y) \left( u(t,x+\ell y) - u(t,x) \right) \right|^2 \, \dee y \, , \\ 
       -\frac{12}{d(d+2)} \cdot D_{L}[u] &= \lim_{\ell\rightarrow 0} {\frac{1}{\ell}} \dashint_{\S^{d-1}} y^j \left( u^j(t,x+\ell y) - u^j (t,x) \right) \left| T_L(y) \left( u(t,x+\ell y) - u(t,x) \right) \right|^2 \, \dee y \, , \\
       -\frac{4(d-1)}{d(d+2)} \cdot D_{T}[u] &= \lim_{\ell\rightarrow 0} \frac 1 \ell \dashint_{\S^{d-1}} y^j \left( u^j(t,x+\ell y) - u^j (t,x) \right) \left| T_T(y) \left( u(t,x+\ell y) - u(t,x) \right) \right|^2 \, \dee y \, . \label{4/15}
    \end{align}
    \end{subequations}
\end{theorem}

\begin{remark}[\textbf{Numerology}] We note that in the case $d=3$, the formulas in~\eqref{formulas} give local, deterministic versions of the $\sfrac 43$, $\sfrac 45$, and $\sfrac{4}{15}$ laws.  The discrepancy between the numbers $\sfrac{4}{15}$ and $\sfrac{8}{15}$ in~\eqref{4/15} is due to the fact that Kolmogorov considered only $u(t,x+\ell\ov y)-u(t,x)$, where $\ov y$ was a particular choice of unit vector perpendicular to $y$; in $3$ dimensions, this constitutes half of the space perpendicular to $y$. 
\end{remark}

\begin{remark}[\textbf{Dissipation measures and energy flux}]
It is natural to interpret $D[u]$, or $D_\bullet[u]$, as the local energy flux of $u$. The solutions constructed by Giri, Kwon, and the author in~\cite{GKN23b, GKN23} then have local energy flux given by a non-negative $L^1_{t,x}$ function (in fact a smooth function).  We remark that in general, the distributional inequality $D[u]\geq 0$ ensures that $D[u]$ is actually a locally finite positive measure, thus justifying the common terminology ``Duchon-Robert measure.''  The laws in~\eqref{formulas} apply as well as to any of the constructions of (nearly) Onsager-critical solutions to 3D Euler due to Isett~\cite{Isett2018}, Buckmaster, De Lellis, Sz\'ekelyhidi, and Vicol~\cite{BDLSV17}, and the author and Vicol~\cite{NV22}, or Giri and Radu~\cite{GR23} for the 2D Euler equations. However, none of the latter set of examples mentioned above satisfy $D[u] \geq 0$, and so $D[u]$ is only a distribution, but not an $L^1_{t,x}$ function or a positive measure.  Furthermore, such solutions cannot arise as an inviscid limit of suitable solutions to the Navier-Stokes equations, which by definition satisfy $D[u^\nu] \geq 0$. For $C^\alpha$ constructions of Euler solutions satisfying $D[u]\geq 0$, where $\alpha$ is however bounded away from the Onsager threshold $\sfrac 13$, we refer to work of Isett~\cite{Is22} and De Lellis and Kwon~\cite{DK22}.
\end{remark}

\begin{remark}[\textbf{Dissipation measures and intermittency}]
    Incompressible fluids for which the turbulent region is not space-filling are referred to as ``intermittent,'' contrasting with the K41 prediction of a turbulent region with full measure.  Intermittent solutions may then belong to $H^\alpha_x$ for some $\alpha>\sfrac 13$, and $C^\beta_x$ for some $\beta<\sfrac 13$; such solutions to 3D Euler were first constructed by Buckmaster, Masmoudi, the author, and Vicol~\cite{BMNV}. There is substantial evidence indicating that physical flows are in fact intermittent and have space-time energy dissipation measure, given by either the formula of Duchon-Robert or any of the formulas in~\eqref{formulas}, concentrating on lower-dimensional sets~\cite{Frisch}. We refer to recent work of De Rosa and Isett~\cite{DeRosaIsett}, De Rosa, Drivas, and Inversi~\cite{DRDI}, Cheskidov and Shvydkoy~\cite{CS}, and references therein for mathematical examples and quantifications of this phenomenon.
\end{remark}

\begin{remark}[\textbf{The $\sfrac 45$ law and self-regularization}]
    We note that in~\cite{Drivas}, Drivas adopts a hypothesis on anti-alignment of velocity increments and \emph{assumes} that~\eqref{formulas} holds in order to prove that an $L^3_{t,x}$ Euler solution enjoys higher regularity.  Theorem~\ref{thm:main} shows that this latter assumption may be removed.
\end{remark}

In order to present our next corollary, we define
\begin{subequations}\label{structures}
    \begin{align}
        S^\nu_I(t,\ell) &= \frac{d}{4}\int_{\T^d}\dashint_{\S^{d-1}} \left( u_\nu(t,x+\ell y) - u_\nu(t,x) \right) \cdot y \left| u_\nu(t,x+\ell y) - u_\nu(t,x)  \right| \, \dee y \, \dee x \, , \\
        S^\nu_L(t,\ell) &= \frac{d(d+2)}{12}\int_{\T^d}\dashint_{\S^{d-1}} \left( u_\nu(t,x+\ell y) - u_\nu(t,x) \right) \cdot y \left| T_L\left(u_\nu(t,x+\ell y) - u_\nu(t,x) \right) \right|^2 \, \dee y \, \dee x \, ,\label{fortyfive} \\
        S^\nu_T(t,\ell) &= \frac{d(d+2)}{4(d-1)}\int_{\T^d}\dashint_{\S^{d-1}} \left( u_\nu(t,x+\ell y) - u_\nu(t,x) \right) \cdot y \left| T_T\left( u_\nu(t,x+\ell y) - u_\nu(t,x) \right)  \right|^2 \, \dee y \, \dee x \, , \label{8:15}
    \end{align}
\end{subequations}
where we have included the appropriate constants in order to streamline the following statement.

\begin{corollary}[\textbf{$\sfrac 45$, $\sfrac 43$, and $\sfrac{4}{15}$ laws in the inviscid limit $\nu\rightarrow 0$}]\label{local:limit}
    Let $u_\nu$, $\nu\in(0,1)$ be Leray-Hopf solutions of the $d$-dimensional forced Navier-Stokes system~\eqref{eqn:NSE:Euler} with a fixed initial datum $u_{0}\in L^2$ such that there exist $\alpha,\sigma>0$ satisfying
\begin{equation}\label{k41regularity}
        {\sup_{\nu \in (0,1)} \left[ \| u_\nu \|_{L^2_t B^\alpha_{2,\infty,x}} + \| f_\nu \|_{L^{1+\sigma}_t L^2_x} \right] < \infty  \, . } 
    \end{equation}
Set
\begin{equation}\label{epsy}
        \varepsilon_\nu(t) = \frac{1}{2} \|u_0\|_{L^2(\T^d)}^2 - \frac 12\|u_\nu(t)\|_{L^2(\T^d)}^2 + \int_0^t \langle f_\nu(r) , u_\nu(r) \rangle \, \dee r \, . 
    \end{equation}
    Then for 
    \begin{equation}\label{correct:lengthscale}
        {\ell_D(\nu) = \nu^{L} \, , \qquad \textnormal{where} \quad L < {\frac{1}{2(1-\alpha)}} } \, ,
    \end{equation}
    and any $p\in [1,\infty)$, we have that
\begin{align}\label{43:law}
        \lim_{\ell_I \rightarrow 0} \limsup_{\nu\rightarrow 0} \sup_{\ell\in[\ell_D, \ell_I]} \int_0^1 \, \left| \int_0^t \frac{S^\nu_\bullet(r,\ell)}{\ell} \, \dee r + \varepsilon_\nu(t) \right|^p \,   \dee t = 0 \, .
\end{align}   
If in addition there exists $\ov \alpha>0$ such that $u^\nu$ is uniformly bounded in $L^\infty_t B^{\ov\alpha}_{2,\infty,x}$, in which case we allow $\sigma=0$ so that $f_\nu$ is bounded in $L^1_t L^2_x$, then 
    \begin{align}
        \lim_{\ell_I \rightarrow 0} \limsup_{\nu\rightarrow 0} \sup_{\ell\in[\ell_D, \ell_I]} \sup_{t\in[0,1]} \left| \int_0^t \frac{S^\nu_\bullet(r,\ell)}{\ell} \, \dee r + \varepsilon_\nu(t) \right| = 0 \, . 
    \end{align}
\end{corollary}
\begin{remark}[\textbf{Commentary on the assumptions of Corollary~\ref{local:limit}}]
    We have assumed in~\eqref{k41regularity} a uniform bound for $u_\nu$ in $L^2_t B^\alpha_{2,\infty,x}$ in analogy with second-order structure function exponents, which pertain to averages in space and time of squared velocity increments. In fact this assumption may be weakened by replacing the uniform Besov bound on $u^\nu$, which bounds a supremum over translations by $|y|>0$ of an appropriate integral, by a uniform-in-$\nu$ bound for the supremum over translations by $|y| \in [\ell_D, \ell_I]$ of the same integral.  We thank T.~Drivas for pointing out this strengthening,
    and for pointing us to his paper with Nguyen~\cite{DrivasNguyen}, which uses this assumption and the same choice of dissipative length scale as in~\eqref{correct:lengthscale} to ensure that weak limits of Leray-Hopf solutions of Navier-Stokes produce weak Euler solutions in $L^2_t B^{\alpha}_{2,\infty,x}$.  We emphasize that the choice of length scale $\ell_D$ depends only on $\alpha$ and not $\ov \alpha$ from the optional assumption. In~\cite{HPZZ}, the authors assume a uniform $L^1_t H^{\alpha}_x$ bound, whichever however implies a uniform $L^2_t H^{\sfrac \alpha 2}_x$ bound by interpolation with the uniform $L^\infty_t L^2_x$ bound implied by the assumption of Leray solutions emanating from a fixed initial datum $u_0\in L^2$.  We remark that the initial data $u_0$ in~\cite{HPZZ} is taken to belong to $H^\beta$ for some $\beta>0$, although this is in fact not necessary for the proof.
\end{remark}
\begin{remark}[\textbf{An example of Kolmogorov's laws in the inviscid limit}]
    In~\cite[Theorem~A]{BCCDLS}, Bru\`e et al construct a sequence of Leray-Hopf solutions $u_\nu$ to the 3D Navier-Stokes equations forced by a family of smooth forces $f_\nu$, emanating from a single initial data $u_0$, which satisfy~\eqref{k41regularity} for $q=3$ and $\alpha=\sfrac 13-$ (actually $L^3_t C^{\sfrac 13-}_x$).  From Corollary~\ref{local:limit}, we therefore have that the sequence satisfies Kolmogorov's laws in a range of length scales (nearly) commensurate with the usual K41 dissipative length scale $\ell_D(\nu) \sim \nu^{\sfrac 34}$.
\end{remark}

% The second corollary refers to a construction of dissipative solutions of 3D Euler due to the author, Giri, and Kwon~\cite{GKN23}.
% \begin{corollary}[Local $\sfrac 45$, $\sfrac 43$, and $\sfrac{4}{15}$ law for $\nu=0$]\label{local:inviscid}
%     \red{Our solutions satisfy some things.}
% \end{corollary}

We prove Theorem~\ref{thm:main} in section~\ref{sec:two}. Then in section~\ref{sec:three}, we prove Corollary~\ref{local:limit}.  Finally, appendix~\ref{sec:app} contains a few technical tools used throughout.

\subsubsection*{Acknowledgments} This work was supported by NSF Grant DMS-2307357. The author is grateful to Michael Novack for helpful discussions.  The author thanks Theodore Drivas for commentary on a draft of this paper.

\section{Proof of Theorem~\ref{thm:main}}\label{sec:two}
Throughout this section, we shall use the following notations.  Set 
\begin{equation}\label{eq:varphi}
\varphi_{\ell,\gamma}=\ell^{-d}\varphi_\gamma(\sfrac \cdot \ell),
\end{equation}
for any radially symmetric, non-negative kernel $\varphi_{\gamma}$ with gradient supported in a neighborhood of size $\gamma$ around $\S^{d-1}$, where $\gamma\in[0,1]$; note that then $\varphi_{\ell,\gamma}$ has gradient supported in a neighborhood of size $\gamma\ell$ around $\S^{d-1}_\ell$. We do not necessarily assume that $\varphi_{\ell,\gamma}$ integrates to $1$, as it will be convenient later to choose a particular kernel which does not have unit mass. Then for $f$ an integrable function, we set
\begin{subequations}\label{increments}
\begin{align}
    u^i_{\bullet, \ell, \gamma}(x) &= \int_{\T^d} u^j(x+y) T^{ij}_\bullet(y) \varphi_{\ell,\gamma}(y) \, \dee y \, , \\
    % u^i_{\bullet, \ell, 0}(x) &= \dashint_{\S^{d-1}} u^j(x+\ell y) T^{ij}_\bullet(y) \, \dee y \, , \qquad \gamma=0 \, , \\
    \left(u^i_\bullet f \right)_{\ell,\gamma}(x) &= \int_{\T^d} u^j(x+y) T^{ij}_{\bullet}(y) f(x+y) \varphi_{\ell,\gamma}(y) \, \dee y \, , \\
    % \left(u^i_\bullet f \right)_{\ell,0}(x) &= \dashint_{\S^{d-1}} u^j(x+y) T^{ij}_{\bullet}(y) f(x+y) \, \dee y \, \qquad \gamma = 0 \, , \\
    % \left(u^i_\bullet u^k_\bullet f \right)_{\ell,\gamma}(x) &= \int_{\T^d} u^j(x+y) T^{ij}_{\bullet}(y) u^n(x+y) T^{nk}_\bullet(y) f(x+y) \varphi_{\ell,\gamma}(y) \, \dee y \, , \qquad \red{\textnormal{used????}} \\ 
    \left(u^i_\bullet u^i_\bullet f \right)_{\ell,\gamma}(x) &= \int_{\T^d} u^j(x+y) T^{ij}_{\bullet}(y) u^i(x+y) f(x+y) \varphi_{\ell,\gamma}(y) \, \dee y     \, .
    % \left(u^i_\bullet u^k_\bullet f \right)_{\ell,0}(x) &= \dashint_{\S^{d-1}} u^j(x+y) T^{ij}_{\bullet}(y) u^n(x+y) T^{nk}_\bullet(y) f(x+y) \, \dee y \, \qquad \gamma = 0 \, . 
\end{align}
\end{subequations}
Note that when $\gamma=0$, $\varphi_{\ell,\gamma}$ is not smooth, and convolution with $\varphi_{\ell,\gamma}$ induces a weighted integral over the ball of radius $\ell>0$. 
% While the above definitions clearly give integrable functions of $x$ when $\gamma>0$, we check in Appendix~\ref{sec:app} when that give integrable functions for $\gamma=0$, and that we have convergence as $\gamma\rightarrow 0$, under reasonable integrability assumptions.
%  Furthermore, we have that for any fixed $\gamma \geq 0$, the left-hand side of~\eqref{eq:main:balance} converges to
%     \begin{align}\label{LHS:main}
%         \partial_t \left(|u|^2 \right) + \partial_j \left( u^j \left( |u|^2 + 2p \right) \right)  - 2 u \cdot f - \nu \left( \Delta (|u|^2) - 2 \nabla u : \nabla u \right)
%     \end{align}
%     as $\ell\rightarrow 0$ in the sense of distributions.    Finally, we have that for any fixed $\gamma\geq 0$, 
%     \begin{align}\label{conv:positive}
%         \lim_{\ell\rightarrow 0} D_{I, \ell,\gamma}[u] = \lim_{\ell\rightarrow 0} D_{L, \ell,\gamma}[u] = \lim_{\ell\rightarrow 0} D_{T, \ell,\gamma}[u] := D[u] \, ,  
%     \end{align}
%     where the convergence is in the sense of distributions to the classical Duchon-Robert measure~\eqref{DR}.
% \red{END PASTE}
We split the proof up into steps, in which we first perform some test function computations in \texttt{Step 0}, before addressing the cases $\bullet=I,L,T$ in the following steps.

\noindent\texttt{Step 0: Test function computations}

\noindent Let $\varphi_{\ell,\gamma,\kappa,\bullet}(y)=\varphi_{\ell,\gamma}(y)c_{\kappa,\bullet}(y)$, where if $\bullet=L,T$, $c_{\kappa,\bullet}$ is smooth, radially symmetric, non-decreasing in $|y|$, takes values in $[0,1]$, and satisfies $c_{\kappa,\bullet}(0)=0$ and $c_{\kappa,\bullet}(y)=1$ for $y\geq\kappa$, where $\kappa\ll \ell,\gamma$, or $c_{\kappa,\bullet}\equiv 1$ if $\bullet =I$. Since the computations in \texttt{Step 0} only require $c_{\kappa,\bullet}$ to be smooth, we shall suppress the dependence on $\bullet$ and write simply $\varphi_{\ell,\kappa,\gamma}$. Note that $T^{ij}_\bullet \varphi_{\ell,\gamma,\kappa}$ is now a smooth kernel for $\gamma>0$ since the possible singularity at $y=0$ has been excised.  For a smooth test function $\phi:[0,T]\times\T^d\rightarrow \R^d$, set
$$ \phi_{\bullet}^i \ast \varphi_{\ell,\ga,\kappa}(t,x) = \int_{\T^d} \phi^j(t,x+y) T_{\bullet}^{ij}(y) \varphi_{\ell,\ga,\kappa}(y) \, \dee y \, , $$
and define versions which depend on $\kappa$ of the quantities in~\eqref{increments} similarly. Now testing~\eqref{eqn:NSE:Euler} with $\phi_\bullet^i \ast \varphi_{\ell,\ga,\kappa}$ using~\eqref{weak:formulation}, changing variables $x+y\rightarrow x$ and $y\rightarrow -y$, and using the radial symmetry of $\varphi_{\ell,\ga,\kappa}$ and $T^{ij}_\bullet$, we have that
\begin{align}
     \int_{[0,T]\times\T^d} &\phi^i_\bullet\ast\varphi_{\ell,\gamma,\kappa}(t,x) f^i(t,x) \, \dee t \, \dee x + \int_{\T^d} \left[ \phi^i_\bullet\ast\varphi_{\ell,\gamma,\kappa}(0,x) u_0^i(x) - \phi^i_\bullet \ast \varphi_{\ell,\gamma,\kappa}(T,x) u^i(T,x) \right] \, \dee x \notag \\
     &=\int_{[0,T]\times\T^d} \int_{\T^d} \varphi_{\ell,\ga,\kappa}(y) T_\bullet^{ij}(y) \bigg{[} -\pa_t \phi^j(t,x) u^i(t,x+y) - \pa_k \phi^j(t,x) (u^k u^i)(t,x+y) \notag\\
    &\qquad \qquad \qquad - \pa_i \phi^j(t,x) p(t,x+y) + \nu \pa_k \phi^i(t,x) \cdot \pa_k u^i(t,x+y) \bigg{]} \, \dee y \,  \dee t \, \dee x \,  . \label{setting:up} 
\end{align}
We claim that we can actually choose $\phi^i=\phi u^i$ in the above computation, where $\phi$ is a scalar-valued test function, and that the integral identity remains valid.  To do so, we must justify passing to the limit $\phi_n^i \rightarrow u^i$ in each term, where $\{\phi_n^i\}_{n\geq 1}$ are smooth test functions approximating $u^i$ in appropriate topologies.  First, we see that if $f\in L^{\sfrac 32}_{t,x}$ or $f\in L^1_t L^2_x$, then the first term from the first line can be bounded using the $L^3_{t,x}$ or the $C^0_t L^2_x$ bound on $u$. For the second term on the first line, we use the $C^0_t L^2_x$ bound on $u$. For the first term following the equals sign (with the time derivative), we note that the above expression implies that $\partial_t u^i_{\ell,\gamma,\kappa}$, in the sense of distributions, is equal to
\begin{align}
     &\int_{[0,T]\times\T^d} \phi^i_\bullet\ast\varphi_{\ell,\gamma,\kappa}(t,x) f^i(t,x) \, \dee t \, \dee x + \int_{\T^d} \left[ \phi^i_\bullet\ast\varphi_{\ell,\gamma,\kappa}(0,x) u_0^i(x) - \phi^i_\bullet \ast \varphi_{\ell,\gamma,\kappa}(T,x) u^i(T,x) \right] \, \dee x \notag \\
    &+\iint_{[0,T]\times\T^d\times \T^d} \varphi_{\ell,\ga,\kappa}(y) T_\bullet^{ij}(y) \bigg{[} - \pa_k \phi^j(t,x) (u^k u^i)(t,x+y) - \pa_i \phi^j(t,x) p(t,x+y) \notag\\
    &\qquad \qquad \qquad \qquad \qquad \qquad \qquad \qquad + \nu \pa_k \phi^i(t,x) \cdot \pa_k u^i(t,x+y) \bigg{]} \, \dee y \,  \dee t \, \dee x \notag 
\end{align}
and therefore is a bounded linear functional on $\phi\in L^3_{t,x} \cap C^0_t L^2_x$ if $\nu=0$, and $\phi \in L^3_{t,x} \cap C^0_t L^2_x \cap L^2_t \dot H^1_x$ if $\nu >0$.  Indeed if $\nu=0$ and $f\in L^{\sfrac 32}_{t,x}$ or $f\in L^1_t L^2_x$, $u_0 \in L^2_x$, $u\in (C^0 L^2_x) \cap L^3_{t,x}$, or if $\nu >0$ and additionally $u\in L^2_t \dot H^1_x$, then the above expression is bounded by the $L^3_{t,x} \cap C^0_t L^2_x$ norm of $\phi$.  Under these assumptions, we then have that~\eqref{setting:up} becomes, upon plugging in $\phi^i=\phi u^i$ and integrating by parts,
\begin{align}
     &\int_{[0,T]\times\T^d} (\phi u^i) \ast \varphi_{\ell,\ga,\kappa}(t,x) f^i(t,x) \, \dee t \, \dee x \label{eq:1} \\
     &\qquad =\int_{[0,T]\times\T^d} \int_{\T^d} \varphi_{\ell,\ga,\kap}(y) T_\bullet^{ij}(y) \bigg{[} \phi(t,x) u^j(t,x) \pa_t u^i(t,x+y) + \phi(t,x) u^j(t,x) \pa_k(u^k u^i)(t,x+y) \notag\\
    &\qquad \qquad \qquad \qquad - \pa_i (\phi(t,x) u^j(t,x)) p(t,x+y) + \nu \pa_k (\phi (t,x) u^j(t,x)) \cdot \pa_k u^i(t,x+y) \bigg{]} \, \dee y \, \dee t \, \dee x \, . \notag 
\end{align}

Next, we claim that we can test~\eqref{eqn:NSE:Euler} with
\begin{align}\notag 
    U^k(t,x) = \phi(t,x) \int_{\T^d} u^i(t,x+y) T^{ik}_\bullet(y) \varphi_{\ell,\ga,\kap}(y) \, \dee y = \phi(t,x) u^k_{\bullet,\ell,\gamma,\kap}(t,x) \, ,
\end{align}
where $\phi$ is any smooth test function. Indeed by the same arguments that allowed us to obtain~\eqref{eq:1}, $\pa_t U^k$ is a bounded linear functional on $L^{3}_{t,x}\cap C^0_t L^2_x$ and $\pa_j U^k \in L^3_{t,x}$, while if $\nu>0$ we additionally have $U^k \in L^2_t \dot H^1_x$. We therefore deduce that
\begin{align}
     \int_{[0,T]\times\T^d} & \phi(t,x) u^k_{\bullet,\ell,\gamma,\kap}(t,x) f^k(t,x) \, \dee t \, \dee x \notag\\
     &\qquad \qquad + \int_{\T^d} \left[\phi(0,x) u^k_{\bullet,\ell,\gamma\,\kap}(0,x)(0,x) u_0^k(x) - \phi(T,x) u^k_{\bullet,\ell,\gamma,\kap}(T,x) u^k(T,x) \right]\, \dee x \notag \\
     &=\int_{[0,T]\times\T^d} \bigg{[} -\pa_t \left( \phi(t,x) \int_{\T^d} u^i(t,x+y) T^{ik}_\bullet(y) \varphi_{\ell,\ga,\kap}(y) \, \dee y \right) u^k(t,x) \notag\\
     &\qquad \qquad \qquad - \pa_j \left( \phi(t,x) \int_{\T^d} u^i(t,x+y) T^{ik}_\bullet(y) \varphi_{\ell,\ga,\kap}(y) \, \dee y \right) (u^k u^j)(t,x) \notag\\
    &\qquad \qquad \qquad - \pa_i \left( \phi(t,x) \int_{\T^d} u^k(t,x+y) T^{ik}_\bullet (y) \varphi_{\ell,\ga,\kap}(y) \, \dee y \right) p(t,x) \notag\\
    &\qquad \qquad \qquad + \nu \pa_k \left( \phi(t,x) \int_{\T^d} u^j(t,x+y) T^{ij}_\bullet (y) \varphi_{\ell,\ga,\kap}(y) \, \dee y \right) \cdot \pa_k u^i(t,x) \bigg{]} \, \dee t \, \dee x \,  . \label{eq:2} 
\end{align}
We now split into cases based on $\bullet=I,L,T$.
\smallskip

\noindent
\texttt{Step 1: $\bullet=I$}

\noindent We first prove that 
    \begin{align}
        \pa_t \left(u^i u^i_{I,\ell,\gamma} \right) &+ \partial_j \left( u^i u^i_{I,\ell,\gamma} u^j + \frac 12 \left(  u^i_I u^i_I u^j \right)_{\ell,\gamma} - \frac 12 \left( u^i_I u^i_I \right)_{\ell,\gamma} u^j \right) + \pa^i \left( p u^i_{I,\ell,\gamma} + p_{\ell,\gamma} u^i \right) \notag\\
        &- u^i f^i_{I,\ell,\gamma} - u^i_{I,\ell,\gamma} f^i - \nu \left( \Delta(u^i u^i_{I,\ell,\gamma}) - 2\pa_k u^i \pa_k u^i_{I,\ell,\gamma} \right) = -2 D_{I,\ell,\gamma}[u] \, , \label{eq:main:balance}
    \end{align}
where 
\begin{align}\label{firstD}
    D_{I,\ell,\gamma}[u]=\frac 14 \int_{\T^d} \pa_j \varphi_{\ell,\ga}(y) \left( u^j(t,x+y) - u^j (t,x) \right) \left| u(t,x+y) - u(t,x) \right|^2 \, \dee y \, , \qquad \gamma>0 \, ,
\end{align}
and if $\gamma=0$, 
\begin{equation}
    D_{I,\ell,0}[u]=\frac{-d}{4\ell}\dashint_{\S^{d-1}} y^j \left( u^j(t,x+\ell y) - u^j(t,x) \right) \left| u(t,x+\ell y) - u(t,x) \right|^2  \, \dee y \, . \notag
\end{equation}
Adding together~\eqref{eq:1} and~\eqref{eq:2} and using that $T_I=\Id$ and $c_{\kappa}\equiv 1$ to simplify, we obtain that 
\begin{align}
     \int_{[0,T]\times\T^d} & \left[(\phi u^k) \ast \varphi_{\ell,\gamma }(t,x) f^k(t,x) + \phi(t,x) u^k_{I,\ell,\gamma }(t,x) f^k(t,x) \right] \, \dee t \, \dee x \notag\\
     &+ \int_{\T^d} \left[\phi(0,x) u^k_{I,\ell,\gamma }(0,x) u_0^k(x) -  \phi(T,x) u^k_{I,\ell,\gamma }(T,x) u^k(T,x)\right] \, \dee x \notag \\
     &=\int_{[0,T]\times\T^d} \bigg{[} -\pa_t \phi(t,x) u^k(t,x) \int_{\T^d} u^k(t,x+y) \varphi_{\ell,\ga}(y) \, \dee y   \notag\\
     &\qquad \qquad \qquad -  \int_{\T^d} \left(\pa_j \phi(t,x) u^i(t,x+y) + \phi(t,x) \pa_j u^i(t,x+y) \right)  \varphi_{\ell,\ga }(y)  (u^i u^j)(t,x) \, \dee y \notag\\
     &\qquad\qquad  + \int_{\T^d} \phi(t,x) u^i(t,x) \varphi_{\ell,\ga }(y)  \pa_j(u^i u^j)(t,x+y) \, \dee y \notag\\
    &\qquad - \int_{\T^d} \varphi_{\ell,\ga }(y) \left(  \pa_i \phi(t,x) u^i(t,x+y)   p(t,x) - \pa_i \phi(t,x) u^i(t,x) p(t,x+y) \right) \, \dee y \notag\\
    &\qquad + \nu (- \pa_{kk} \phi(t,x) u^j(t,x) u^j_{I,\ell,\ga } + 2 \phi(t,x) \pa_k u^j_{I,\ell,\ga }(t,x) \pa_k u^j(t,x) \bigg{]} \, \dee t \, \dee x \,  . \notag
\end{align}
Now we add 
\begin{equation}\label{eq:adding}
 \frac 12 \int_{[0,T]\times\T^d} \phi(t,x) \left[ \pa_j \left( u^j |u_I|^2 \right)_{\ell,\gamma } - u^j \pa_j \left( |u_I|^2 \right)_{\ell,\gamma }  \right] \, \dee t \, \dee x 
\end{equation}
to both sides, note that
\begin{align}
 \int_{\T^d} -\pa_{j}\varphi_{\ell,\gamma }(y) f(x+y) \, \dee y = \int_{\T^d} \varphi_{\ell,\gamma }(y) \pa_{j} f(x+y) \, \dee y &= \int_{\T^d} \varphi_{\ell,\gamma }(y) \pa_{j} f(x-y) \, \dee y \notag\\ 
 &= \int_{\T^d} \pa_j \varphi_{\ell,\gamma }(y)  f(x-y) \, \dee y\notag\\
 &= \pa_j f_{\ell,\gamma } \label{mollies}
\end{align}
due to the radial symmetry of $\varphi_{\ell,\gamma}$,
and rearrange to deduce that
\begin{align}
    \int_{[0,T]\times\T^d} &\bigg{[} -\pa_t \phi  \left( u^k u^k_{I,\ell,\gamma } \right) - \partial_j \phi \left( u^i_{I,\ell,\gamma } u^i u^j + \frac 12 \left( \left( u^j |u_I|^2 \right)_{\ell,\gamma } - u^j \left( |u_I|^2 \right)_{\ell,\gamma } \right) \right) - \nu \Delta \phi \left( u^j u^j_{I,\ell,\gamma } \right)  \notag\\
    &\qquad + 2 \nu \phi \pa_k u^j_{I,\ell,\gamma } \pa_k u^j - \pa_i\phi (u^i_{I,\ell,\gamma } p + u^i p_{I,\ell,\gamma }) \bigg{]} \, \dee t \, \dee x \notag\\
    &=\int_{[0,T]\times\T^d} \left[(\phi u^k) \ast \varphi_{\ell,\gamma }(t,x) f^k(t,x) + \phi(t,x) u^k_{I,\ell,\gamma }(t,x) f^k(t,x) \right] \, \dee t \, \dee x \notag\\
     &\quad + \int_{\T^d} \left[\phi(0,x) u^k_{I,\ell,\gamma }(0,x) u_0^k(x) -  \phi(T,x) u^k_{I,\ell,\gamma }(T,x) u^k(T,x)\right] \, \dee x \notag \\
     &\quad +\int_{\T^d\times[0,T]} \phi  \left[ \pa_j u^i_{I,\ell,\gamma } u^i u^j  - u^i \pa_j \left( u^j_I u^i_I \right)_{\ell,\gamma } + \frac 12 \pa_j \left( \left( u^j |u_I|^2 \right)_{\ell,\gamma } - u^j \left( |u_I|^2 \right)_{\ell,\gamma } \right) \right] \, \dee t \, \dee x \, . \label{last:one}
\end{align}
We now consider the last term on the right-hand side.  After using that $\div u = 0$ and~\eqref{mollies} to simplify the expression
\begin{align}\label{DR}
    -\frac 12 \int_{\T^d\times[0,T]} \int_{\T^d} \phi(t,x) \pa_j \varphi_{\ell,\gamma }(y) \left( u^j(t,x+y) - u^j(t,x) \right) \left| u(t,x+y) - u(t,x) \right|^2  \, \dee y \, \dee t \, \dee x 
\end{align}
coming from~\eqref{firstD} multiplied by $-2$, we find that the last term on the right-hand side of~\eqref{last:one} is in fact equal to~\eqref{DR}, completing the proof of~\eqref{eq:main:balance} when $\gamma>0$.

% In order to prove~\eqref{LHS:main} when $\bullet=I$ and $\gamma>0$, we use the $L^3_{t,x}$ integrability of $u$, the $L^{\sfrac 32}_{t,x}$ integrability of $p$ (which follows from the Calder\'on-Zygmund inequality), the $L^2_{t,x}$ integrability of $\nabla u$ if $\nu>0$, and the $L^1_t L^2_x$ or $L^{\sfrac 32}_{t,x}$ integrability of $f$.

We now work to prove~\eqref{eq:main:balance} when $\gamma=0$, which requires passing to the limit $\gamma\rightarrow 0$ in~\eqref{last:one} and~\eqref{DR}. We first pass to the limit in every term from~\eqref{last:one} except the very last term (which is now~\eqref{DR}), using the integrability assumptions on all involved quantities and the dominated convergence theorem.  Now in order to pass to the limit $\gamma\rightarrow 0$ in~\eqref{DR}, we will use that $\varphi_{\ell,\gamma}(y) = \ell^{-d} \varphi_\gamma(|y|/\ell)$, where $\varphi_\gamma$ is smooth, positive, integrates to $1$, and has gradient supported in a $\gamma$-neighborhood around the sphere of radius $1$. Then changing to spherical variables $y\rightarrow (r,\sigma)$, we rewrite~\eqref{DR} (ignoring the $-\sfrac 12$ prefactor) as
\begin{align}
    & \int_{\T^d} \int_0^T \int_{0}^\infty \int_{\S^{d-1}} \phi(t,x) \frac{\varphi_{\gamma}'\left(\frac{r}{\ell}\right)r^{d}}{r\ell^{1+d}} \sigma_j \left( u^j(t,x+r \sigma ) - u^j(t,x) \right) \left| u(t,x+r\sigma) - u(t,x) \right|^2  \, \dee \sigma \, \dee r \, \dee t \, \dee x \notag\\
    &= \int_{\T^d}\int_0^T\int_{0}^\infty \phi(t,x) \frac{\varphi_{\gamma}'\left(r'\right) (r')^d }{\ell r'} \int_{\S^{d-1}} \sigma_j \left( u^j(t,x+ \ell r' \sigma ) - u^j(t,x) \right) \notag\\
    &\qquad \qquad \qquad \qquad \qquad \qquad \times \left| u(t,x+\ell r'\sigma) - u(t,x) \right|^2  \, \dee \sigma \, \dee r' \, \dee t \, \dee x \, . \notag 
\end{align}
When $\gamma\rightarrow 0$, we use~\eqref{dumber:bound} to pass to the limit and obtain
% use that for any smooth vector field $U(x)$,
% $$  \int_{\T^d} \nabla \varphi_{\ell,\gamma}(y) \cdot U(y) \, \dee y \qquad \xlongrightarrow[\gamma\rightarrow 0^+] \qquad\qquad  - \frac{d}{\ell} \dashint_{\S^{d-1}} U(\ell y) \cdot y \, \dee y \, . $$
% Applying this equality with a rescaled (by $-\sfrac{d}{\ell}$) version of~\eqref{dumber:bound}, we have that
that~\eqref{DR} converges to
\begin{equation}
    \frac{d}{2\ell} \int_{\T^d\times[0,T]}\dashint_{\S^{d-1}} \phi(t,x) y^j \left( u^j(t,x+\ell y) - u^j(t,x) \right) \left| u(t,x+\ell y) - u(t,x) \right|^2  \, \dee y \, \dee t \, \dee x \, , \notag
\end{equation}
concluding the proof of~\eqref{eq:main:balance} for $\gamma=0$. In order to prove~\eqref{eq:main:balance:thm}, we have that the left-hand side of~\eqref{eq:main:balance} converges in the sense of distributions as $\ell\rightarrow 0$ to the left-hand side of~\eqref{eq:main:balance:thm}, which guarantees that $\lim_{\ell\rightarrow 0}D_{I,\ell,0}[u]= D[u]$ by the uniqueness of distributional limits, concluding the proof of Theorem~\ref{thm:main} for $\bullet=I$.

\smallskip
\noindent\texttt{Step 2: $\bullet = L, T$}

\noindent Using that $T_L=\frac{y \otimes y}{|y|^2}$ and $T_T = T_I - T_L$, we first simplify the terms from~\eqref{eq:1} and~\eqref{eq:2} involving the pressure above by noting that for the radially symmetric kernels $\varphi_{\ell,\gamma,\kappa,\bullet}=\varphi_{\ell,\gamma} c_{\kappa,\bullet}$, 
\begin{align*}
     \partial_i \left( T^{ik}_L(y)\varphi_{\ell,\ga,\kappa,L}(y) \right) &= \partial_i \left( \frac{y^i y^k}{|y|^2} \varphi_{\ell,\ga,\kappa,L}(y) \right) = \partial_k \left( \varphi_{\ell,\ga,\kappa,L}(y) - (d-1) \int_{|y|}^{\infty} \frac{\varphi_{\ell,\ga,\kappa,L}(|\ov y|)}{|\ov y|} \, \dee \ov y \right) \\ 
     &=: \partial_k \Psi_{\ell,\gamma,\kappa,L}\\
    \partial_i \left( T^{ik}_T(y)\varphi_{\ell,\ga,\kappa,L}(y) \right) &= \partial_i \left( \left( T^{ik}_I(y)  - T^{ik}_L(y) \right) \varphi_{\ell,\gamma,\kappa,T} \right) =: \partial_k \Psi_{\ell,\gamma,\kappa,T}
\end{align*}
are gradients of potentials. Using this to simplify and adding together~\eqref{eq:1} and~\eqref{eq:2}, we find that
% \begin{align}
%     % \Psi_{\ell,\gamma,T}(y) = (d-1) \int_{|y|}^{\infty} \frac{\varphi_{\ell,\ga}(|\ov y|)}{|\ov y|} \, \dee \ov y \, , \qquad \qquad 
%     \Psi_{\ell,\gamma,L} = \varphi_{\ell,\ga}(y) - (d-1) \int_{|y|}^{\infty} \frac{\varphi_{\ell,\ga}(|\ov y|)}{|\ov y|} \, \dee \ov y \, .
% \end{align}
\begin{align}
     &\int_{[0,T]\times\T^d}  \left[\phi u^k(t,x) f_{\bullet,\ell,\gamma,\kappa}^k(t,x) + \phi(t,x) u^k_{\bullet,\ell,\gamma,\kappa}(t,x) f^k(t,x) \right] \, \dee t \, \dee x \notag\\
     &\qquad+ \int_{\T^d} \left[\phi(0,x) u^k_{\bullet,\ell,\gamma,\kappa}(0,x) u_0^k(x) -  \phi(T,x) u^k_{\bullet,\ell,\gamma,\kap}(T,x) u^k(T,x)\right] \, \dee x \notag \\
     &\quad =\int_{[0,T]\times\T^d} \bigg{[} -\pa_t \phi(t,x) u^k(t,x) \int_{\T^d} u^i(t,x+y) T^{ik}_\bullet(y) \varphi_{\ell,\ga,\kappa,\bullet}(y) \, \dee y   \notag\\
     &\qquad \qquad -  \int_{\T^d} \left(\pa_j \phi(t,x) u^i(t,x+y)  + \phi(t,x) \pa_j u^i(t,x+y) \right) T^{ik}_\bullet(y) \varphi_{\ell,\ga,\kap,\bullet}(y)  (u^k u^j)(t,x) \, \dee y \notag\\
     &\qquad\qquad  + \int_{\T^d} \phi(t,x) u^i(t,x) T^{ik}_\bullet(y) \varphi_{\ell,\ga,\kap,\bullet}(y)  \pa_j(u^k u^j)(t,x+y) \, \dee y \notag\\
    &\qquad - \int_{\T^d} \left(  \pa_i \phi(t,x) u^k(t,x+y) T^{ik}_\bullet (y) \varphi_{\ell,\ga,\kap,\bullet}(y)  p(t,x) - \pa_j \phi(t,x) u^j(t,x) \Psi_{\ell,\ga,\kap,\bullet}(y) p(t,x+y) \right) \, \dee y \notag\\
    &\qquad + \nu (- \pa_{kk} \phi(t,x) u^j(t,x) u^j_{\ell,\ga,\kap,\bullet} + 2 \phi(t,x) \pa_k u^j_{\ell,\ga,\kap,\bullet}(t,x) \pa_k u^j(t,x) \bigg{]} \, \dee t \, \dee x \,  . \label{withkappa}
\end{align}
We now pass to the limit $\kappa\rightarrow 0$ using the integrability assumptions on $u$ and $p$ and the dominated convergence theorem, obtaining that an identical version of~\eqref{withkappa} holds, with the $\kappa$ however removed.  Next, we add the analogue of~\eqref{eq:adding}, but with $\bullet=L,T$ instead of $\bullet=I$, to both sides, obtaining (after abbreviating the convolution of $\Psi_{\ell,\gamma,\bullet}$ with $p$ by $p_{\ell,\gamma,\bullet}$)
\begin{align}
    &\int_{[0,T]\times\T^d} \bigg{[} -\pa_t \phi  \left( u^k u^k_{\bullet,\ell,\gamma } \right) - \partial_j \phi \left( u^i_{\bullet,\ell,\gamma } u^i u^j + \frac 12 \left( \left( u^j |u_\bullet|^2 \right)_{\ell,\gamma } - u^j \left( |u_\bullet|^2 \right)_{\ell,\gamma } \right) \right) - \nu \Delta \phi \left( u^j u^j_{\bullet,\ell,\gamma } \right)  \notag\\
    &\qquad\qquad + 2 \nu \phi \pa_k u^j_{\bullet,\ell,\gamma } \pa_k u^j - \pa_i\phi (u^i_{\bullet,\ell,\gamma } p + u^i p_{\ell,\gamma,\bullet}) \bigg{]} \, \dee t \, \dee x \notag\\
    &\qquad=\int_{[0,T]\times\T^d} \left[\phi u^k (t,x) f_{\bullet,\ell,\gamma}^k(t,x) + \phi(t,x) u^k_{\bullet,\ell,\gamma }(t,x) f^k(t,x) \right] \, \dee t \, \dee x \notag\\
    &\qquad\quad + \int_{\T^d} \left[\phi(0,x) u^k_{\bullet,\ell,\gamma }(0,x) u_0^k(x) -  \phi(T,x) u^k_{\bullet,\ell,\gamma }(T,x) u^k(T,x)\right] \, \dee x \notag \\
    &\qquad\quad +\int_{\T^d\times[0,T]} \phi  \left[ \pa_j u^k_{\bullet,\ell,\gamma } u^k u^j  - u^i \pa_j \left( u^j u^i_\bullet \right)_{\ell,\gamma } + \frac 12 \pa_j \left( \left( u^j |u_\bullet|^2 \right)_{\ell,\gamma } - u^j \left( |u_\bullet|^2 \right)_{\ell,\gamma } \right) \right] \, \dee t \, \dee x \, . \label{last:one:L}
\end{align}
Now by direct computation, using that $\div \, u=0$, $\langle u, T_\bullet u \rangle = \langle T_\bullet u, T_\bullet u \rangle$, and the spherical symmetry of $\varphi_{\ell,\gamma}T_\bullet$ and anti-symmetry of its gradient $\nabla (\varphi_{\ell,\gamma}T_\bullet)$, we may rewrite the last term from~\eqref{last:one:L} as
\begin{align}
    &\int_{\T^d\times[0,T]} \phi  \left[ \pa_j u^k_{L,\ell,\gamma } u^k u^j  - u^i \pa_j \left( u^j u^i_L \right)_{\ell,\gamma } + \frac 12 \pa_j \left( \left( u^j |u_L|^2 \right)_{\ell,\gamma } - u^j \left( |u_L|^2 \right)_{\ell,\gamma } \right) \right] \, \dee t \, \dee x \notag \\
    &\quad=- \frac 12 \int_{\T^d\times[0,T]}\int_{\T^d} \phi \, \partial_{y_k} \left( T^{ij}_L \varphi_{\ell,\gamma} \right)\bigg{[}  \left( u^i(t,x+y) - u^i(t,x) \right) \notag\\
    &\qquad \qquad \qquad \left( u^j(t,x+y) - u^j(t,x) \right) \left( u^k(t,x+y) - u^k(t,x) \right) \bigg{]} \, \dee y \, \dee t \, \dee x \notag \\
    &\quad=- \frac 12 \int_{\T^d\times[0,T]} \int_{\T^d} \phi \, \bigg{[} \nabla \varphi_{\ell,\gamma} \cdot \left( u(t,x+y) - u(t,x) \right) \left| \left( u(t,x+y) - u(t,x) \right) T_L \right|^2 \notag \\
    &\qquad \qquad \qquad \qquad + \frac{2y}{|y|^2}\varphi_{\ell,\gamma} \cdot \left( u(t,x+y) - u(t,x) \right) \left| \left( u(t,x+y) - u(t,x) \right) T_T \right|^2 \bigg{]} \, \dee y \, \dee t \, \dee x \label{mess:one}
\end{align}
if $\bullet =L$, and
\begin{align}
    &\int_{\T^d\times[0,T]} \phi  \left[ \pa_j u^k_{T,\ell,\gamma } u^k u^j  - u^i \pa_j \left( u^j u^i_T \right)_{\ell,\gamma } + \frac 12 \pa_j \left( \left( u^j |u_T|^2 \right)_{\ell,\gamma } - u^j \left( |u_T|^2 \right)_{\ell,\gamma } \right) \right] \, \dee t \, \dee x \notag \\
    &=-\frac 12\int_{\T^d\times[0,T]}\int_{\T^d} \phi \, \partial_{y_k} \left( T^{ij}_T \varphi_{\ell,\gamma} \right)\bigg{[}  \left( u^i(t,x+y) - u^i(t,x) \right) \notag\\
    &\qquad \qquad \qquad \left( u^j(t,x+y) - u^j(t,x) \right) \left( u^k(t,x+y) - u^k(t,x) \right) \bigg{]} \, \dee y \, \dee t \, \dee x \notag \\
    &=-\frac 12\int_{\T^d} \int_0^T \int_{\T^d} \phi \, \left( \nabla \varphi_{\ell,\gamma} - \frac{2y}{|y|^2} \varphi_{\ell,\gamma} \right) \left( u(t,x+y) - u(t,x) \right) \left| \left( u(t,x+y) - u(t,x) \right) T_T \right|^2 \, \dee y \, \dee t \, \dee x \label{ODE}
\end{align}
if $\bullet=T$. 

We first work to prove an analogue of~\eqref{eq:main:balance}, but for $\bullet=L$.  We pass to the limit $\gamma\rightarrow 0$ in~\eqref{last:one:L} and the last line of~\eqref{mess:one}, obtaining for the latter
\begin{align*}
    &\frac{d}{2\ell}\int_{\T^d\times[0,T]} \dashint_{\S^{d-1}} \phi(t,x) \, y\cdot \left( u(t,x+\ell y) - u(t,x) \right) \left| T_L(y) \left( u(t,x+\ell y) - u(t,x) \right) \right|^2 \, \dee y \, \dee t \, \dee x \\
    & - \int_{\T^d}\int_0^T \int_{\T^d}  \phi(t,x) \frac{\mathbf{1}_{B_\ell(0)}(y)}{|B_\ell(0)| |y|^2} y\cdot \left( u(t,x+y) - u(t,x) \right) \left| \left( u(t,x+y) - u(t,x) \right) T_T \right|^2 \, \dee y \, \dee t \, \dee x \, .
\end{align*}
We wish to eliminate the second term in the above expression, so that we obtain an energy balance with the proper third-order longitudinal structure function on the right-hand side. To do so, we choose in~\eqref{ODE} and~\eqref{last:one:L}
$$ \ov\varphi_{\ell,\gamma}(y) = \frac{1}{|B_\ell(0)|} \left( \frac{|y|^2}{\ell^2} - 1 \right) \mathbf{1}_{\{0 \leq |y| \leq \ell\}}(y) \, , $$
which is a smooth function except at $|y|=\ell$, where it is however continuous; this choice may be justified by an application of the dominated convergence theorem. Note that 
$$ \partial_k \ov\varphi_{\ell,\gamma}(y) - \frac{2y_k}{|y|^2} \ov\varphi_{\ell,\gamma}(y) = \mathbf{1}_{\{0\leq |y| \leq \ell\}} (y) \frac{1}{|B_\ell(0)|} \left( \frac{2y_k}{\ell^2} - \frac{2y_k}{|y|^2}\left( \frac{|y|^2}{\ell^2}-1 \right) \right) = \frac{2y_k}{|B_{\ell}(0)| |y|^2} \mathbf{1}_{B_\ell(0)} (y) \, .  $$
We use this choice of $\ov\varphi_{\ell,\gamma}$ in~\eqref{last:one:L} for $\bullet=T$ and subtract the resulting balance from~\eqref{last:one:L} with $\bullet=L$ and $\gamma=0$, obtaining that
\begin{align}
    &\int_{[0,T]\times\T^d} \bigg{[} -\pa_t \phi  \left( u^k u^k_{L,\ell} \right) - \partial_j \phi \left( u^i_{L,\ell} u^i u^j + \frac 12 \left( \left( u^j |u_L|^2 \right)_{\ell} - u^j \left( |u_L|^2 \right)_{\ell} \right) \right) - \nu \Delta \phi \left( u^j u^j_{L,\ell} \right)  \notag\\
    &\qquad\qquad + 2 \nu \phi \pa_k u^j_{L,\ell} \pa_k u^j - \pa_i\phi (u^i_{L,\ell} p + u^i p_{L,\ell}) \bigg{]} \, \dee t \, \dee x \notag\\
    &\qquad\qquad-\int_{[0,T]\times\T^d} \left[\phi u^k (t,x) f_{L,\ell}^k(t,x) + \phi(t,x) u^k_{L,\ell}(t,x) f^k(t,x) \right] \, \dee t \, \dee x \notag\\
    &\qquad\qquad - \int_{\T^d} \left[\phi(0,x) u^k_{L,\ell}(0,x) u_0^k(x) -  \phi(T,x) u^k_{L,\ell}(T,x) u^k(T,x)\right] \, \dee x \notag \\
    &=  \frac{d}{2\ell} \int_{\T^d\times[0,T]} \dashint_{\S^{d-1}} \phi(t,x) \,  y \cdot \left( u(t,x+\ell y) - u(t,x) \right) \left| T_L(y) \left( u(t,x+\ell y) - u(t,x) \right) \right|^2 \, \dee y \, \dee t \, \dee x \, . \label{last:one:L:L}
\end{align}
where
\begin{equation}
    u^i_{L,\ell}(t,x) := \int_{\T^d} \left( \frac{1}{|B_\ell(0)|} \mathbf{1}_{\{ 0\leq |y|\leq \ell \}}(y) T_L^{ij}(y) - \ov\varphi_{\ell,\gamma}(y) T_T^{ij}(y) \right) u^j(t,x+y) \, \dee y \, , \notag 
\end{equation}
and $p_{L,\ell}$ is defined analogously.  In order to pass to the limit on both sides of~\eqref{last:one:L:L}, we first claim that for any $L^p$ vector field $g^k$, 
\begin{align}
    \lim_{\ell\rightarrow 0} \int_{\T^d} \left| g^k_{L,\ell}(x) - \frac{3d}{d(d+2)}  g^k(x) \right|^p \, \dee x =  0 \, . \notag
\end{align}
To prove this, we use that 
$\int_{\T^d} \ov\varphi_{\ell,\gamma}(y) \, \dee y = \frac{-2}{d+2}$, and $\int_{\T^d} T^{ij}_T(y) \, \dee y = \frac{d-1}{d}$, so that the integral of $\varphi_{\ell,\gamma}T^{ij}_T=\frac{-2(d-1)}{d(d+2)} \delta^{ij}$; in addition, we use that $\int_{\T^d} T^{ij}_L(y) \, \dee y = \frac{\delta^{ij}}{d}$. Computing similarly for $p_{L,\ell}$, combining these results, and passing to the limit $\ell\rightarrow 0$ in~\eqref{last:one:L:L}, we obtain that the left-hand side converges to $\frac{3d}{d(d+2)}$ multiplied by the left-hand side of~\eqref{eq:main:balance}.  Dividing the factor of $\frac{d}{2}$ on the right-hand side of~\eqref{last:one:L:L} by twice $\frac{3d}{d(d+2)}$ concludes the proof of~\eqref{fortyfive}.

In order to prove~\eqref{8:15}, we use that $|T_I v|^2 = |T_L v|^2 + |T_T v|^2$ for any vector $v$.  Then writing the coefficient on the left-hand side of~\eqref{4/15} as $\frac{1}{d C_T}$ for $C_T$ undetermined, we find $C_T$ must solve
$$  \frac{12}{d(d+2)} D_L + \frac{1}{dC_T} D_T = \frac{4}{d} D_I \qquad \underset{\eqref{eq:main:balance:thm}}{\implies} \qquad \frac{3}{d+2} + \frac{1}{4C_T} = 1 \qquad \iff \qquad C_T = \frac{d+2}{4(d-1)} \, . $$

\section{Proof of Corollary~\ref{local:limit}}\label{sec:three}
We treat only the $\sfrac 43$ law (or $\sfrac 4d$ in $d$ dimensions), using~\eqref{eq:main:balance} with $\gamma=0$. The proof of the $\sfrac{4}{5}$ law follows identically using~\eqref{last:one:L:L}, and the proof of the $\sfrac{4}{15}$ law follows again from additivity. Applying the distributional equality~\eqref{eq:main:balance} with $\gamma=0$, a test function $\phi(t,x)=\mathbf{1}_{[0,T]}(t)$, using $u_{\ell,\nu}$ to denote the average of $u_\nu$ on a ball of radius $\ell\in [\ell_D, \ell_I]$, and recalling~\eqref{structures} and~\eqref{epsy}, we have that
\begin{align}
    \int_0^T \frac{S_I^\nu(r,\ell)}{\ell} \, \dee r + \varepsilon_\nu(T) 
    &= \frac 12 \int_{\T^d} \left[ u_0^2(x) - u^k_{\ell,\nu}(x) u^k_0(x) + u_{\ell,\nu}^k(T,x) u^k_\nu(T,x) - |u_\nu(T,x)|^2 \right] \, \dee x \notag\\
    & \quad + \int_{\T^d\times[0,T]}  f_\nu(t,x) \left[ u_\nu(t,x) - u_{\ell,\nu}(t,x) \right] \, \dee t \, \dee x \notag\\
    & \quad + \int_{\T^d\times[0,T]} \nu \pa_k u_{\ell,\nu}^j \pa_k u_\nu^j \, \dee t \, \dee x \, . \label{terms}
\end{align}
Examining the first term from~\eqref{terms}, we may bound it by
\begin{align}\label{bound:1}
 &\left\| \dashint_{B_1(0)} u^k_0(x) \left( u^k_0(x+\ell y) - u^k_0(x) \, \dee y \right) \right\|_{L^1} + \left\| \dashint_{B_1(0)} u^k_\nu(T,x) \left( u^k_\nu(T,x+\ell y) - u^k_\nu(T,x) \, \dee y \right) \right\|_{L^1}  \, . \notag
\end{align}
The first term approaches zero as $\ell\rightarrow 0$ due to continuity of the integral for $L^2(\T^d)$ functions, while the second may be bounded by 
\begin{align} &\qquad \lesssim \| u_\nu(T) \|_{L^2(\T^d)} \ell^\alpha \| u_\nu(T) \|_{B^{\alpha}_{2,\infty}(\T^d)} \, .
\end{align}
Assuming that we have the optional uniform $L^\infty_t B^{\ov\alpha}_{2,\infty,x}$ bound for $u_\nu$, this term goes to zero as $\ell_I, \ell_D$ go to zero, no matter the precise choice of $\ell_D$. In the case the optional bound is not satisfied, we may interpolate the $L^\infty_tL^2_x$ and $L^2_t B^\alpha_{2,\infty,x}$ bounds to find, for any $p\in [1,\infty)$, an $\alpha(p)< \alpha$ such that $u_\nu$ is uniformly bounded in $L^{p}_t B^{\alpha(p)}_{2,\infty,x}$.  Then fixing $p$ in~\eqref{43:law}, using the uniform-in-$\nu$ $L^p_t B^{\alpha(p)}_{2,\infty,x}$ bound, and integrating~\eqref{bound:1} raised to the $p^{\rm th}$ power with respect to $T$ proves that this term goes to zero as $\ell \rightarrow 0$. 

For the second term from~\eqref{terms}, we may bound it by 
\[
 \| f_\nu \|_{L^1_t L^2} \ell^{\ov\alpha} \| u_\nu \|_{L^\infty_t B^{\ov\alpha}_{2,\infty}} \, , \qquad \textnormal{or} \qquad  \| f_\nu \|_{L^{1+\sigma}_t L^2_x} \ell^\alpha \| u_\nu \|_{L^{\frac{\sigma+1}{\sigma}}_t B^{\alpha}_{2,\infty}} \, .
\]
In either case, we have that the limit as $\ell_I, \ell_D\rightarrow 0$ of this term is zero, uniformly in $T$.  

Thus it remains only to treat the final term. By straightforward computations, we have that
\begin{align*}
    &\left| \int_{\T^d\times[0,T]} \nu \pa_k u_{\ell,\nu}^j \pa_k u_\nu^j \, \dee t \, \dee x  \right| \notag \\ 
    &\qquad \les \limsup_{\gamma\rightarrow 0} \left| \int_{\T^d}\int_0^T \int_{\T^d} \nu \partial_k \varphi_{\ell,\gamma}(y) \left[ u_{\ell,\gamma,\nu}(x-y) - u_{\ell,\gamma,\nu}(x) \right] \partial_k u^j_{\nu}(x) \, \dee y \, \dee t \, \dee x \right| \notag\\
    &\qquad \les \nu^{\sfrac 12} \left\| \nabla u_\nu \right\|_{L^2_{t,x}} \nu^{\sfrac 12} \ell^{-1} \left\| u_\nu \right\|_{L^2_t B^{\alpha}_{2,\infty,x}} \ell^{\alpha} \, .
\end{align*}
By the assumption that $\ell_D = \nu^L$ with $L<\frac{1}{2(1-\alpha)}$ from~\eqref{correct:lengthscale} and the uniform bounds from~\eqref{k41regularity}, we find that the above quantity tends to zero as $\nu\rightarrow 0$, uniformly in $T$.

\appendix
\section{Appendix}\label{sec:app}

\begin{proposition}\label{spheres} Let $f\in L^p(\T^d)$, and let $\sigma \in \S^{d-1}$ with $\dee \sigma$ the normalized surface measure on $\S^{d-1}$. Then for any $0 \leq \ell < 1$,
\begin{align}
    \tilde f(x) := \int_{\S^{d-1}} f(x+\ell \sigma) \, \dee \sigma  \label{eq:dumbest:bound}
\end{align}
is an integrable function of $x$ with $L^p$ norm bounded by $\| f \|_{L^p(\T^d)}$. Furthermore, if $\Psi_\gamma:(-\gamma, \gamma)\rightarrow[0,\infty)$ for $\gamma>0$ are smooth, even functions with unit $L^1$ norm defining a sequence of approximate identities $\{\Psi_{\gamma}\}_{\gamma > 0}$ and $\phi:\T^d\rightarrow \R$ belongs to $L^{\frac{p}{p-1}}(\T^d)$, we have that
\begin{align}
    \lim_{\gamma\rightarrow 0} \mathcal{I}_{\ell,\gamma,\phi}(f) &:= \lim_{\gamma\rightarrow 0} \int_{-\infty}^\infty \Psi_\gamma\left( r\right) \int_{\T^d} \phi(x) \int_{\S^{d-1}} f\left(x+(\ell+r)\sigma\right) \, \dee \sigma \, \dee x \, \dee r \notag\\
    &= \int_{\T^d} \int_{\S^{d-1}} \phi(x) f(x+\ell\sigma) \, \dee \sigma \, \dee x := \mathcal{I}_{\ell,\phi} (f) \, , \label{dumber:bound} \\
    \lim_{\ell\rightarrow 0} \mathcal{I}_{\ell,\phi} (f) &= \int_{\T^d} f(x) \phi(x) \, \dee x \, . \label{capital:D:dumb}
\end{align}
\end{proposition}
\begin{proof}
It is clear that all the claims hold for smooth functions. Considering~\eqref{eq:dumbest:bound} for arbitrary $f\in L^p(\T^d)$, we have that $f(x+\ell\sigma)$ is measurable on the product $\T^d\times\S^{d-1}$, and we can apply Tonelli's theorem to find that
$$ \| f(x+\ell\sigma) \|^p_{L^p(\T^d\times\S^{d-1})} = \int_{\S^{d-1}} \int_{\T^{d}} |f(x+\ell \sigma)|^p \, \dee x \, \dee \sigma = \| f \|^p_{L^p(\T^d)} \, . $$
Now from Fubini's theorem and Jensen's inequality, we have that the projection $\tilde f(x)$ is a measurable function of $x$ with $L^p$ norm no larger than $\|f \|_{L^p(\T^d)}$, as desired. Next, in order to prove~\eqref{dumber:bound}, let $f\in L^p(\T^d)$ and $f_n$ be a smooth approximant of $f$ in $L^p(\T^d)$.  Then we have that
\begin{align*}
    \limsup_{\ga\rightarrow 0} \left|  \mathcal{I}_{\ell,\gamma,\phi}(f) - \mathcal{I}_{\ell,\phi}(f) \right| &\leq \limsup_{\ga\rightarrow 0} \left|  \mathcal{I}_{\ell,\gamma,\phi}(f_n) - \mathcal{I}_{\ell,\phi}(f_n) \right| + \limsup_{\ga\rightarrow 0} \left|  \mathcal{I}_{\ell,\gamma,\phi}(f-f_n) \right| \notag\\
    &\qquad + \limsup_{\ga\rightarrow 0} \left| \mathcal{I}_{\ell,\phi}(f_n-f)  \right|
\end{align*}
The first term above goes to zero, and the latter two are bounded by $\| f - f_n \|_{L^p(\T^d)} \| \phi \|_{L^{\frac{p}{p-1}}(\T^d)}$ after performing a change of variables and applying H\"older's inequality. A completely analogous argument shows that~\eqref{capital:D:dumb} holds.

\end{proof}

\end{document}